\documentclass[reqno]{amsart}
\usepackage{amsmath}
\usepackage{amssymb}
\usepackage{amsthm}
\allowdisplaybreaks[4]

\makeatletter
 
  \@addtoreset{equation}{section}
 \makeatother

\def \ds{\displaystyle}
\def \wt{\widetilde}

\def \R{\mathbb{R}}
\def \C{\mathbb{C}}

\def \v{\varphi}
\def \p{\rho}
\def \f{\phi}

\def \e{\varepsilon}

\def \D{\Delta}
\def \A{\forall}
\def \E{\exists}
\def \pa{\partial}
\def \n{\nabla}

\def \a{\alpha}

\def \me{\mathcal{E}}
\def \n{\nabla}
\def \t{\theta}
\def \c{\chi}
\def \P{\Phi}
\def \ta{\tau}

\newcommand{\im}{\operatorname{Im}}
\newcommand{\re}{\operatorname{Re}}
\theoremstyle{definition}
\newtheorem{dei}{Definition}[section]
\newtheorem{thm}{Theorem}[section]

\newtheorem{lem}{Lemma}[section]
\newtheorem{cor}{Corolary}[section]
\newtheorem{rem}{Remark}[section]

\newtheorem*{nota}{Notation}

\begin{document}
\title[The derivation of the conservation laws]{The derivation of the conservation law for defocusing nonlinear Schr{\"o}dinger equations with non-vanishing initial data at infinity}
\author[H. Miyazaki]{Hayato MIYAZAKI}
\address[]{Department of Mathematics, Graduate School of Science, Hiroshima University, Higashi-Hiroshima, 739-8521, JAPAN}
\email{h-miyazaki@hiroshima-u.ac.jp}

\begin{abstract}
For nonlinear Schr{\"o}dinger equations in less than or equal to four dimension, with non-vanishing initial data at infinity, a new approach to derive the conservation law is obtained. Since this approach does not contain approximating procedure, the argument is simplified and some of technical assumption of the nonlinearity to derive the conservation law and time global solutions, is removed.   
\end{abstract}

\maketitle 

\section{Introduction}
In this paper, we consider defocusing nonlinear Schr{\"o}dinger equations in dimension $n \leq 4$.
\begin{align}
\left\{
\begin{array}{l}
\ds i\frac{\partial u}{\partial t}+\Delta u+f(|u|^2)u=0,\quad t\in(0,T),\;x\in\R^{n}, \\
u(0,x)=u_{0}(x),\quad x\in\R^{n},
\end{array}
\right. \label{gp}
\end{align}
where $u(t,x): (0,T)\times \R^{n} \rightarrow \C$. The initial data $u_{0}$ has the following boundary condition:
\[
|u_{0}(x)|^{2} \rightarrow \p_{0} \quad \text{as} \quad |x| \rightarrow \infty,
\]
where $\p_{0}>0$ denotes the light intensity of the background. The nonlinear term $f$ is assumed to be defocusing. Namely the real-valued function $f$ satisfies the following assumption:
\begin{flalign*}
& \hspace{4cm} f(\p_{0})=0,\quad f'(\p_{0})<0.  & (\textbf{H}_{f})
\end{flalign*}
Equation (\ref{gp}) with non-vanishing initial data at infinity appears as a relevant model in great various physical problems: for example, Bose-Einstein
condensation and superfluidity (see \cite{bib021}, \cite{bib022}, \cite{bib023}), and nonlinear topics (dark solitons, optical vortices) (see \cite{bib025}, \cite{bib060}). Two important model cases for (\ref{gp}) have been extensively studied both in the physical and mathematical literatures: the
Gross-Pitaevskii equation (where $f(r)=1-r$) and the so-called "cubic-quintic" Schr{\"o}dinger
equation (where $f(r)=(r-\p_{0})(2a+\p_{0}-3r)$, $0<a<\p_{0}$). Gallo \cite{bib007} has considered the Cauchy problem for (\ref{gp}). He proved the  following theorem: 
\begin{thm}[Theorem 1.1 in Gallo \cite{bib007}] \label{gallo:01}
Let $n \leq 4$ and $\p_{0}>0$. Assume that $f\in C^{k}(\R_{+})$ ($k=3$ if $n=2$, $3$, $k=4$ if $n=4$) satisfying $(\textbf{H}_{f})$, and there exist $\a_{1} \geq 1$, with a supplementary condition $\a_{1} <\a_{1}^{*}$ if $n=3$, $4$ ($\a_{1}^{*}=3$ if $n=3$, $\a_{1}^{*}=2$ if $n=4$), and $\a_{2}\in \R$ with $\a_{1}-\a_{2} \leq 1/2$ such that
\[
\hspace{3mm} \E{C_{0}>0},\; \E{A}>\p_{0}\; s.t.\; \left\{
\begin{array}{lr}
\A{r} \geq 1,\; \left\{
  \begin{array}{l}
   |f''(r)| \leq C_{0}r^{\a_{1}-3} \; \text{if}\; n=1, 2, 3, \\
   |f'''(r)| \leq C_{0}r^{\a_{1}-4} \; \text{if}\; n=4, 
  \end{array}
 \right. & \hspace{3.5mm} (\textbf{H}_{\a_{1}}) \\
 \left\{ 
  \begin{array}{l}
   \text{if}\; \a_{1} \leq 3/2,\; V\; \text{is bounded from below}, \\
   \text{if}\; \a_{1} > 3/2, \; \A{r} \geq A,\; r^{\a_{2}}\leq C_{0}V(r), 
  \end{array}
 \right. & \hspace{3.5mm} (\textbf{H}_{\a_{2}}) \\
\end{array}
\right.
\]
where $V(r):=\int_{r}^{\rho_{0}}f(s)ds$. Then for any function $\f$ satisfying
\begin{flalign*}
& \hspace{1.5cm} \f\in C_{b}^{k+1}(\R^{n}),\quad \n \f \in H^{k+1}(\R^{n})^{n},\quad |\f|^{2}-\p_{0} \in L^{2}(\R^{n}), & (\textbf{H}_{\f})
\end{flalign*}
(\ref{gp}) is globally well-posed in $\f+H^{1}(\R^{n})$. Namely, for any $w_{0}\in H^{1}(\R^{n})$, there exists an unique $w \in C(\R, H^{1}(\R^{n}))$ such that $\f+w$ is the solution to (\ref{gp}) with the initial data $w(0)=w_{0}$. Moreover, The solution depends continuously on the initial data $w_{0}\in H^{1}$.  
\end{thm}
Generally, we take two steps to construct a time global solution for the Cauchy problem of usual nonlinear Schr{\"o}dinger equations ((NLS)s) (see \cite{bib006}). The first step is to construct a time local solution to Duhamel's integral equation by using a contraction argument. The next step is to extend the solution to the time global solution by using conservation laws. For Cauchy problem (\ref{gp}), we follow the same steps stated above. Thus, to get time global solutions, it is important to obtain conservation laws. We obtain formally the conservation law of energy by multiplying the equation (\ref{gp}) by $\bar{u}_{t}$, integrating over $\R^{n}$, and taking the real part. There are basically two methods to justify the procedure above. One is that solutions is approximated by a sequence of regular solutions, using the continuous dependence of solutions on the initial data. Other is to use a sequence of regularized equations of (\ref{gp}) whose solutions have enough regularities to perform the procedure above (see \cite{bib002}). However, these two methods involve a limiting procedure on approximate solutions. Instead, for (NLS)s with a local interaction nonlinearity, Ozawa \cite{bib002} derive conservation laws by using additional properties of solutions provided by Strichartz estimates. We need the following definitions to mention it:
\begin{dei}
\begin{enumerate}
\item A positive exponent $p'$ is called the dual exponent of $p$ if $p$ and $p'$ satisfy $1/p+1/p'=1$. \\
\item A pair of two exponents $(p,q)$ is called an admissible pair if $(p.q)$ satisfies .
\[
\frac{2}{p}+\frac{n}{q}=\frac{n}{2},\quad p \geq 2,\quad (p,q) \neq (2,\infty). 
\]
\end{enumerate}
\end{dei}
Strichartz estimates are described as the following lemma:
\begin{lem}[Strichartz estimates \cite{bib006}]
Let $(p_{1},q_{1})$ and $(p_{2},q_{2})$ be admissible pairs. Then
\begin{enumerate}
\item for all $f\in L^{2}(\R^{n})$,
\[
\|e^{it\D}f\|_{L^{p_{1}}(\R, L^{q_{1}}(\R^{n}))} \leq C\|f\|_{L^{2}(\R^{n})},
\]
\item let $T>0$, for all $f\in L^{p_{1}'}([0,T],L^{q_{1}'}(\R^{n}))$,
\[
\left\|\int_{-\infty}^{t}e^{i(t-\ta)\D}f(\ta)d\ta \right\|_{L^{p_{2}}([0,T],L^{q_{2}}(\R^{n}))} \leq C\|f\|_{L^{p'_{1}}([0,T],L^{q'_{1}}(\R^{n}))},
\]
where $p'_{1}$ and $p'_{2}$ are the dual exponents of $p_{1}$ and $p_{2}$, respectively.
\end{enumerate}
\end{lem}
In this paper, for the equation (\ref{gp}) in $n=2$, $3$, $4$, we derive the conservation law for time local solutions without approximating procedure. Instead of that, we use Ozawa's idea \cite{bib002}. Note that when $n=1$,  because $H^{1} \hookrightarrow L^{\infty}$, Gallo \cite{bib007} derived it without approximating procedure (see section 2, 3 in Gallo \cite{bib007}), and that for $n \geq 2$, Gallo \cite{bib007} derive it using the approximate argument (see section 5 in Gallo \cite{bib007}). We follow Ozawa's idea, however, we can not derive the conservation law only by Ozawa's idea, due to the nonlinearity and the space of a solution. We derive the conservation law to combine Ozawa's idea with decomposing the nonlinear term by applying the method for the decomposition of Schr{\"o}dinger operator in G{\'e}rard \cite{bib001}. Moreover, we remove some of technical assumptions of the nonlinearity necessary to derive the conservation law. Our main result is as follows.
\begin{thm} \label{thm:01}
Let $n=2$, $3$, $4$. Let $\rho_{0}>0$, and $f\in C^{2}(\R_{+})$ satisfying $(\textbf{H}_{f})$.
Moreover, we assume that there exist $\a_{1} \geq 1$, with a supplementary condition $\a_{1} <\a_{1}^{*}$ if $n=3$, $4$ ($\a_{1}^{*}=3$ if $n=3$, $\a_{1}^{*}=2$ if $n=4$) such that
\begin{flalign*}
& \hspace{2cm} \E{C_{0}>0},\; s.t.\; \A{r} \geq 1,\; |f^{(k)}(r)| \leq C_{0}r^{\a_{1}-1-k}\; (k=1,2). & (\textbf{H}_{\a_{1}}')
\end{flalign*}
Let $\f$ be a function satisfying 
\begin{flalign*}
& \hspace{2cm} \f\in C_{b}^{2}(\R^{n}),\quad \n \f \in H^{2}(\R^{n})^{n},\quad |\f|^{2}-\p_{0} \in L^{2}(\R^{n}). & (\textbf{H}'_{\f})
\end{flalign*}
(Note that such function $\f$ is called as a regular function of finite energy.) Let $w \in C([0,T],H^{1}(\R^{n}))$  be a mild solution of the integral equation
\begin{align}
 w(t)=e^{it\D}w_{0}- i \int_{0}^{t}e^{i(t-t')\D}F(w(t'))dt' \label{eq:01}
\end{align}
for some $w_{0}\in H^{1}$ and $T>0$, where $F(w):=-\D \f-f(|\f+w|^2)(\f+w)$.

Then $\me(w(t))=\me(w_{0})$ for all $t \in [0,T]$, where
\[
\me(w):=\int_{\R^{n}}|\n (\f+w)|^{2}dx+\int_{\R^{n}}V(|\f+w|^{2})dx,
\]
and 
\[
V(r):=\int_{r}^{\rho_{0}}f(s)ds.
\]
\end{thm}
\begin{rem}
Gallo \cite{bib007} prove the energy conservation law under $f\in C^{k}(\R_{+})$ ($k=3$ if $n=2$, $3$, $k=4$ if $n=4$) satisfying ($\textbf{H}_{f}$), $(\textbf{H}_{\a_{1}})$ and $(\textbf{H}_{\a_{2}})$ for some $\a_{1} \geq 1$ and $\a_{2}\in \R$ with $\a_{1}-\a_{2} \leq 1/2$, and $\f$ satisfying $(\textbf{H}_{\f})$, but we can prove it under $f\in C^{2}(\R_{+})$ with ($\textbf{H}_{f}$) and $(\textbf{H}'_{\a_{1}})$ for some $\a_{1} \geq 1$, and $\f$ with $(\textbf{H}'_{\f})$.
\end{rem}
\begin{rem}
For proofs of the a priori estimate of $f(|\f+w|^2)(\f+w)$ (that is Lemmas \ref{lem01} - \ref{lem04}, and Lemmas 4.1 - 4.4 in Gallo \cite{bib007}) and boundedness of $H^{1}$ norm of $w$ on bounded intervals (that is Lemma 3.3 in Gallo \cite{bib007}), we need that there exists $C_{\a_{1}}>0$ such that for any $r \geq 0$, 
\begin{align}
r^{1/2}|f^{(k)}(r)| \leq C_{\a}(1+r^{\max(0, \a_{1}-(2k+1)/2)})\quad (k=1,2), \label{ff:01}
\end{align}
where $1 \leq \a_{1}$ with the same supplementary condition in Theorem \ref{gallo:01}.
If $3/2<\a_{1} \leq 2$, then we can not obtain (\ref{ff:01}) from $(\textbf{H}_{\a_{1}})$ or $(\textbf{H}_{\a_{2}})$. Therefore by replacing $(\textbf{H}_{\a_{1}})$ with $(\textbf{H}_{\a_{1}}')$, we deduce (\ref{ff:01}) from $(\textbf{H}_{\a_{1}}')$ only. To show (\ref{ff:01}), we do not need $(\textbf{H}_{\a_{2}})$.
\end{rem}

Moreover, as a corollary to the main result, we can deduce a globally well-posedness of (\ref{gp}). Due to theorem \ref{thm:01}, we can remove a technical assumption of the nonlinear term. We have the following result:
\begin{cor}
Let $n=2$, $3$, $4$. We assume that $f$ and $\f$ satisfy the same assumptions as in Theorem \ref{thm:01}, with a supplementary assumption as $f$ satisfying $(\textbf{H}_{\a_{2}})$ for some $\a_{2}\in \R$ with $\a_{1}-\a_{2} \leq 1/2$. Then (\ref{gp}) is globally well-posed in $\f+H^{1}(\R^{n})$. That is, for any  $w_{0}\in H^{1}(\R^{n})$, there exist a unique $w \in C(\R, H^{1}(\R^{n}))$ such that $\f+w$ solves (\ref{gp}) with the initial data $w(0)=w_{0}$. Moreover, for any $T>0$, the flow map $w_{0} \mapsto w \; (H^{1} \rightarrow C([0,T], H^{1}))$ is Lipschitz continuous on the bounded sets of $H^{1}(\R^{n})$. The energy $\me(w)$ is conserved by the flow.
\end{cor}
The structure of this paper is as follows. In section 2, we introduce the previous results of Gallo \cite{bib007} on the local existence of solutions of (\ref{gp}). In section 3 and 4, we give estimates of the nonlinear term and results of the time-derivative term needed for the proof of the main result, respectively. In section 5, we prove the main result.
\begin{nota}
For a Banach space $X$, $T>0$ and $p \in [1, \infty]$, $L_{T}^{p}X$ denotes the Banach space $L^{p}([0,T], X)$ equipped with its natural norm.
\end{nota}

\section{Previous results}
For $n \leq 4$, Gallo \cite{bib007} prove the globally well-posedness of (\ref{gp}). We state the result for $n=2$, $3$, $4$. A first strategy of the proof is that (\ref{gp}) is transformed as follows to look for a solution of (\ref{gp}) under the form $\f+w$.
\begin{align}
\left\{
\begin{array}{l}
\ds i\frac{\partial w}{\partial t}+\Delta w=F(w(t)),\quad t\in (0,T),\;x\in\R^{n}, \\
w(0,x)=w_{0}(x),\quad x\in\R^{n},
\end{array}
\right. \label{gp1}
\end{align}
where
\[
F(w):=-\D \f-f(|\f+w|^2)(\f+w).
\]

In a next strategy, he prove that (\ref{gp1}) is locally well-posed in $H^{1}$ by using Strichartz estimates and a contraction argument for the map 
\[
 \qquad \P(w) =e^{it\D}w_{0}-i\int_{0}^{t}e^{i(t-s)\D}F(w(s))ds,
\]
in the space
\[
X_{T}:=L^{\infty}_{T}H^{1} \cap L^{p}_{T}W^{1,q}
\] 
equipped with its natural norm 
\[
\|w\|_{X_{T}}:=\|w\|_{L^{\infty}_{T}H^{1}}+\|w\|_{L^{p}_{T}W^{1,q}},
\]
where a pair $(p,q)$ is a admissible pair defined as $(p,q):= (6/n , 6)$ for $n=2$, $3$, $(p,q):= (2,4)$ for $n=4$.
We remark that Gallo \cite{bib007} takes $(p, q) := (4, 4)$ for $n=2$. Note that our choice also works for getting local existence of solution to (\ref{gp}). For locally well-posedness, Gallo \cite{bib007} prove the following theorem:
\begin{thm}[Gallo \cite{bib007}]\label{thm:1}
Let $n=2$, $3$, $4$. Let $\rho_{0}>0$, and $f\in C^{2}(\R_{+})$ satisfying $(\textbf{H}_{f})$.
Moreover, we assume that there exist $\a_{1} \geq 1$, with a supplementary condition $\a_{1} <\a_{1}^{*}$ if $n=3$, $4$ ($\a_{1}^{*}=3$ if $n=3$, $\a_{1}^{*}=2$ if $n=4$), and $\a_{2}\in \R$ with $\a_{1}-\a_{2} \leq 1/2$ such that $(\textbf{H}_{\a_{1}})$ and $(\textbf{H}_{\a_{2}})$. Let $\f$ be a function satisfying $(\textbf{H}_{\f}')$. 

Then for any $R>0$, there exists $T(R)>0$ such that for any $w_{0} \in H^{1}$ with $\|w_{0}\|_{H^{1}} \leq R$, there exists an unique solution $w \in X_{T(R)}$ of the integral equation (\ref{eq:01}). Moreover $w \in C([0,T(R)],H^{1})$. 

If $\wt{w} \in C([0,T],H^{1})$ solves (\ref{eq:01}) for some $T>0$, then $\wt{w} \in X_{T}$, and $\wt{w} \in X_{T}$ is the unique solution to (\ref{eq:01}) in $C([0,T],H^{1})$.

Also the flow map is locally Lipschitz continuous on the bounded sets of $H^{1}$, indeed for any $R>0$, there exists $T(R)>0$ such that for any $T' \in (0,T(R)]$ and $w_{0}$, $\wt{w}_{0} \in H^{1}$ with $\|w_{0}\|_{H^{1}} \leq R$ and $\|\wt{w}_{0}\|_{H^{1}} \leq R$, corresponding solutions $w$, $\wt{w} \in X_{T'}$ of (\ref{eq:01}) satisfy the following locally Lipschitz continuity: 
\begin{align}
 \|w-\wt{w}\|_{X_{T'}} \leq C\|w_{0}-\wt{w}_{0}\|_{H^{1}}, \label{thm2f}
\end{align}
where $C$ is a positive constant depending on $\|w\|_{X_{T'}}$ and $\|\wt{w}\|_{X_{T'}}$. Especially, for the same constant $C$, 
\[
 \|w-\wt{w}\|_{L^{\infty}_{T'}H^{1}} \leq C\|w_{0}-\wt{w}_{0}\|_{H^{1}}.
\]
Furthermore, the energy $\me(w(t))$ is conserved for all $t \in [0,T]$. 
\end{thm}

\begin{rem}
To obtain the local existence theorem above, it seems to be too much to assume both $(\textbf{H}_{\a_{1}})$ and $(\textbf{H}_{\a_{2}})$ for some $\a_{1} \geq 1$ and $\a_{2}\in \R$ with $\a_{1}-\a_{2} \leq 1/2$. Theorem \ref{thm:1} can be shown only assuming $(\textbf{H}_{\a_{1}}')$. To show only the local existence theorem, we do not need $(\textbf{H}_{\a_{2}})$.  
\end{rem}

From Theorem \ref{thm:1}, a local solution of (\ref{gp1}) is constructed as the following theorem:
\begin{thm} \label{thm:2}
Let $n=2$, $3$, $4$. Let $w_{0}\in H^{1}$. Let $T>0$ and let $w$ be a mild solution of the integral equation (\ref{eq:01}) with $w \in C([0,T],H^{1})$.
Then, for any $t_{0}\in [0,T]$, there exists $v(t_{0}) \in H^{-1}$ such that
\[
\frac{w(t_{0}+h)-w(t_{0})}{h} \rightarrow v(t_{0})\quad in\ H^{-1}\ as\ h \rightarrow 0.
\]
Moreover, denoting $v(t_{0})$ by $\pa_{t}w(t_{0})$, $w$ is a solution of (\ref{gp1}), indeed $w$ satisfies
\begin{enumerate}
\item $i\pa_{t}w(t)+\D w(t) = F(w(t))$ in $H^{-1}$ for all $t \in [0,T]$,
\item $w(0)=w_{0}$.
\end{enumerate}
\end{thm}

\section{The estimates of nonlinear terms}
In what follows, we put $\wt{F}(w) = -f(|\f+w|^{2})(\f+w)$. Applying directly the decomposition of $F(w)$ that Gallo \cite{bib007} used, we can deduce the following decompositions for $\wt{F}(w)$. Note that we can show Lemmas \ref{lem01} - \ref{lem04} by applying the same method to $\wt{F}(w)$ as corresponding lemmas for $F(w)$ in Gallo \cite{bib007}. The statements of lemma \ref{lem01} and Lemma \ref{lem03} is slightly different from these lemmas in Gallo \cite{bib007}. Therefore we prove them in appendix.
\begin{lem} \label{lem01}
Let $T>0$. For any $w \in X_{T}$, there exist
\[
  \wt{F}_{1}(w) \in L^{\infty}_{T}L^{2}, \qquad \wt{F}_{2}(w) \in L^{\infty}_{T}L^{q'}
\]
such that 
\[
  \wt{F}(w)=\wt{F}_{1}(w)+\wt{F}_{2}(w).
\]
Moreover it follows that 
\begin{align}
&\|\wt{F}_{1}(w)\|_{L^{\infty}_{T}L^{2}}+\|\wt{F}_{2}(w)\|_{L^{\infty}_{T}L^{q'}}  \nonumber \\
&\quad  \leq C(1+\|w\|_{L^{\infty}_{T}L^{2}})+C(\|w\|_{L^{\infty}_{T}H^{1}}^2+\|w\|_{L^{\infty}_{T}H^{1}}^{\max(2, 2\a_{1}-1)}),
\end{align}
where $C$ is a positive constant depending on $T$. 
Also for a same decomposition of $\wt{F}(w)$ in the above, we have $\wt{F}_{2}(w) \in L^{p}_{T}L^{2}$ and  
\begin{align*}
\|\wt{F}_{2}(w)\|_{L^{p}_{T}L^{2}} &\leq C(\|w\|^{2}_{L^{\infty}_{T}H^{1}}+\|w\|^{\max(2, 2\a_{1}-1)}_{X_{T}}),  
\end{align*}
where $C$ is a positive constant depending on $T$. Thus $F_{2}(w) \in L^{p}_{T}L^{2}$.
\end{lem}
\begin{lem} \label{lem02}
Let $T>0$. For any $w \in X_{T}$, there exist
 \[
  \wt{G}_{1}(w) \in L^{\infty}_{T}L^{2}, \qquad \wt{G}_{2}(w) \in L^{p'}_{T}L^{q'}.
 \]
such that
 \[
  \n \wt{F}(w)=\wt{G}_{1}(w)+\wt{G}_{2}(w).
 \]
Moreover it follows that 
\begin{align*}
&\|\wt{G}_{1}(w)\|_{L^{\infty}_{T}L^{2}}+\|\wt{G}_{2}(w)\|_{L^{p'}_{T}L^{q'}} \\
&\leq C(1+\|\n w\|_{L^{\infty}_{T}L^{2}}) \\
 &\quad +C(1+\|\n w\|_{L^{\infty}_{T}L^{2}})(\|w\|_{L^{\infty}_{T}H^{1}}+\|w\|_{X_{T}}^{\max(1,2\a_{1}-2)}),
\end{align*}
where $C$ is a positive constant depending on $T$. 
\end{lem}
\begin{lem}\label{lem03}
Let $T>0$. For any $w_{1}$, $w_{2} \in X_{T}$, decomposing $f(|\f+w|^{2})(\f+w)$ as Lemma \ref{lem01}, it follows that
\begin{align*}
 &\|\wt{F}_{1}(w_{1})-\wt{F}_{1}(w_{2})\|_{L^{\infty}_{T}L^{2}}+\|\wt{F}_{2}(w_{1})-\wt{F}_{2}(w_{2})\|_{L^{\infty}_{T}L^{q'}} \\
  & \leq C\|w_{1}-w_{2}\|_{L^{\infty}_{T}H^{1}}+C\|w_{1}-w_{2}\|_{L^{\infty}_{T}H^{1}} \\
  &\quad \times ((\|w_{1}\|_{L^{\infty}_{T}H^{1}}+\|w_{2}\|_{L^{\infty}_{T}H^{1}})+(\|w_{1}\|_{L^{\infty}_{T}H^{1}}+\|w_{2}\|_{L^{\infty}_{T}H^{1}})^{\max(1,2\a_{1}-2)}),
 \end{align*}
where $C$ is a positive constant depending on $T$.
\end{lem}
\begin{lem}\label{lem04}
Let $T>0$. For any $w_{1}$, $w_{2} \in X_{T}$, decomposing $f(|\f+w|^{2})(\f+w)$ as Lemma \ref{lem02}, it follows that
\begin{align*}
 &\|\wt{G}_{1}(w_{1})-\wt{G}_{1}(w_{2})\|_{L^{\infty}_{T}L^{2}}+\|\wt{G}_{2}(w_{1})-\wt{G}_{2}(w_{2})\|_{L^{p'}_{T}L^{q'}} \\
 &\leq C\|\n (w_{1}-w_{2})\|_{L^{\infty}_{T}L^{2}} \\
 &\quad +C(1+\|w_{1}\|_{L^{\infty}_{T}H^{1}}+\|w_{2}\|_{L^{\infty}_{T}H^{1}})^{\max(1,2\a_{1}-2)}\|w_{1}-w_{2}\|_{L^{\infty}_{T}H^{1}} \\
 &\quad +C\|w_{1}-w_{2}\|_{X_{T}}(\|w_{1}\|_{X_{T}}^{\max(1,2\a_{1}-2)}+\|w_{2}\|_{X_{T}}^{\max(1,2\a_{1}-2)}) \\
 &\quad +C\|w_{1}-w_{2}\|_{X_{T}}(1+\|w_{1}\|_{L^{\infty}_{T}H^{1}}+\|w_{2}\|_{L^{\infty}_{T}H^{1}}) \\
 &\quad \times (\|w_{1}\|_{X_{T}}^{\max(0,2\a_{1}-3)}+\|w_{2}\|_{X_{T}}^{\max(0,2\a_{1}-3)}), 
 \end{align*}
where $C$ is a positive constant depending on $T$. 
\end{lem}
\begin{rem}\label{lem:cf}
Let $T>0$. Lemma \ref{lem01} and Sobolev embedding $H^{1} \hookrightarrow L^{q}$, imply that for any $w \in C([0, T], H^{1})$ and $t \in [0, T]$, $F(w(t)) \in H^{-1}$. Furthermore, for any $t_{0} \in [0, T]$, Lemma \ref{lem03} yields   
\begin{align*}
\|F(w(t))-F(w(t_{0}))\|_{H^{-1}} &\leq C \|w(t)-w(t_{0})\|_{H^{1}} \\
&\rightarrow 0 \quad as \quad t \rightarrow t_{0},
\end{align*}
where $C$ is a positive constant depending on $\|w\|_{L^{\infty}_{T}H^{1}}$. To show it, for $w \in C([0, T], H^{1})$, it suffices to put $w_{1}(s)=w(t)$ and $w_{2}(s)=w(t_{0})$ ($0 \leq s \leq T$). Thus we also obtain $F(w) \in C([0, T], H^{-1})$.
\end{rem}

In the proof of the main result, we use the following Lemma: 
\begin{lem} \label{lemf05}
For any $\eta \in L^{2}+L^{q'}$, it follows that
\begin{align}
\|\c(D_{x})\eta\|_{H^{1}} \leq C\|\eta\|_{L^{2}+L^{q'}}. \label{lem05}
\end{align}
Moreover for any $\eta \in {\mathcal S}'(\R^{n})$ with $\n \eta \in L^{2}+L^{q'}$, we obtain
\begin{align}
\|(1-\c(D_{x}))\eta\|_{H^{1}} \leq C\|\n \eta\|_{L^{2}+L^{q'}}. \label{lem06}
\end{align}
\end{lem}
Note that if $X$ and $Y$ are Banach spaces, then $X+Y$ is a Banach space equipped with the norm
\[
\|v\|_{X+Y} := \inf\{\|v_{1}\|_{X}+\|v_{2}\|_{Y}\; : \; v=v_{1}+v_{2}, \; v_{1} \in X, \; v_{2} \in Y\}.
\]

We use the following Theorem to prove Lemma \ref{lemf05}.

\begin{thm}[Fourier multiplier theorem \cite{bib010}] 
Let $1<p<\infty$. For some integer $s>n/2$, suppose that $m(\xi) \in C^{s}(\R^{n}\backslash\{0\}) \cap L^{\infty}(\R^{n})$. Assume also that for all multi-index $\a$ with $|\a| \leq s$, there exists a positive constant $C_{\a}$ such that
\begin{align}
|\pa_{\xi}^{\a}m(\xi)| \leq C_{\a}|\xi|^{-|\a|}. \quad (\xi \in \R^{n} \setminus \{0\}) \label{mik}
\end{align}
Then, there exists a positive constant $C$ depending on $p,\ C_{\a},\ d,\ s$ such that 
\[
\|m(D_{x})f\|_{L^{p}} \leq C\|f\|_{L^{p}}.
\]
\end{thm}

\begin{proof}[Proof of Lemma \ref{lemf05}]
For any $\eta \in L^{2}+L^{q'}$, There exist $\eta_{1} \in L^{2}$ and $\eta_{2} \in L^{q'}$ such that $\eta=\eta_{1}+\eta_{2}$. $Q(\xi)$ denotes $(1+|\xi|^{2})\c(\xi)$. Also, $Q(\xi)$ satisfies (\ref{mik}) since $\c \in C_{0}^{\infty}(\R^{n})$. Therefore, using Fourier multiplier theorem, we obtain
\begin{align*}
\|\c(D_{x})\eta\|_{H^{1}}&= \|Q(D_{x})\eta\|_{H^{-1}} \\
 &\leq \|Q(D_{x})\eta_{1}\|_{L^{2}}+\|Q(D_{x})\eta_{2}\|_{L^{q'}} \\
 &\leq C\left( \|\eta_{1}\|_{L^{2}}+\|\eta_{2}\|_{L^{q'}}\right).  
\end{align*}
Therefore, we deduce that
\begin{align*}
\|\c(D_{x})\eta\|_{H^{1}} \leq C\|\eta\|_{L^{2}+L^{q'}}.
\end{align*}
Next, for any $\eta \in {\mathcal S}'(\R^{n})$ with $\n \eta \in L^{2}+L^{q'}$, there exist $(\zeta_{1}^{j}(w))_{j=1, \cdots , n} \in L^{2}$ and $(\zeta_{2}^{j}(w))_{j=1, \cdots , n} \in L^{q'}$ such that $\n \eta = \zeta_{1}+\zeta_{2}$. Using $P_{j}(\xi) := -i\xi_{j}/|\xi|^2$ \; ($\xi:=(\xi_{j})_{j=1, \cdots, n} \in \R^{n}$), we have
\begin{align*}
(1-\c(D_{x}))\eta &= (1-\c(D_{x}))\sum_{j=1}^{n}P_{j}(D_{x})\pa_{j}\eta \\
 &= \sum_{j=1}^{n} (1-\c(D_{x}))P_{j}(D_{x})\zeta_{1}^{j} + \sum_{j=1}^{n}(1-\c(D_{x}))P_{j}(D_{x})\zeta_{2}^{j}.
\end{align*}
Fourier multiplier theorem implies
\begin{align*}
&\|(1-\c(D_{x}))\eta\|_{H^{1}+W^{1,q'}} \\
&\leq \sum_{j=1}^{n} \|(1-\c(D_{x}))P_{j}(D_{x})\zeta_{1}^{j}\|_{H^{1}}+\sum_{j=1}^{n}\|(1-\c(D_{x}))P_{j}(D_{x})\zeta_{2}^{j}\|_{W^{1,q'}} \\
&\leq C\|\zeta_{1}\|_{L^{2}}+C\|\zeta_{2}\|_{L^{q'}}.
\end{align*}
Thus we get
\begin{align*}
\|(1-\c(D_{x}))\eta\|_{H^{1}+W^{1,q'}} &\leq C\|\n \eta\|_{L^{2}+L^{q'}}.
\end{align*}
\end{proof}

\section{Regularities of time-derivative term}
In this section, we shall show properties of the time-derivative term $\pa_{t} u$.
\begin{lem}\label{delt:1}
Let $n=2$, $3$, $4$, and let $(p,q):= (6/n , 6)$ for $n=2$, $3$, $(p,q):= (2,4)$ for $n=4$. Let $w$ be a solution of equation (\ref{gp1}) belonging to $C([0,T],H^{1})$ for some $T>0$ with the initial data $w(0)=w_{0} \in H^{1}$. Then for any $0 < \e < T' <T$,
\begin{flalign*}
&\text{(i)} \quad \ds \left\|\frac{w(\cdot+h)-w(\cdot)}{h}-\pa_{t}w(\cdot)\right\|_{C([\e,T'], H^{-1})} \rightarrow 0 \quad as \quad h \rightarrow 0, &\\
\text{and}& &\\
&\text{(ii)} \quad \ds \left\|\frac{w(\cdot+h)-w(\cdot)}{h}-\pa_{t}w(\cdot)\right\|_{L^{p}([\e,T'],W^{-1,q})} \rightarrow 0 \quad as \quad h \rightarrow 0.&
\end{flalign*}
\end{lem}
\begin{proof}
Note that equation (\ref{gp1}) implies 
\begin{align}
\pa_{t}w &= i(\D w-F(w)) \label{td1}.
\end{align}
We show (i) and (ii) using (\ref{td1}). \\
\underline{Proof of (i).} \quad Note that from Theorem \ref{thm:2}, for any $0 \leq t \leq T$, $\pa_{t}w(t) \in H^{-1}$ exists in strong sense. Hence, it suffices to show continuity of $\pa_{t}w(t)$ on $[0, T]$. Clearly,  
\begin{align}
\|\D w\|_{H^{-1}} \leq \|\n w\|_{L^{2}}, \label{td2}
\end{align}
which yields $\D w\in C([0,T], H^{-1})$. Using (\ref{td1}), (\ref{td2}) and Remark \ref{lem:cf}, we obtain
\begin{align*}
\pa_{t}w \in C([0,T], H^{-1}). 
\end{align*}
Hence, it follows that for all $t_{0}$, $t \in [0,T]$,
\begin{align}
w(t)-w(t_{0})=\int_{t_{0}}^{t}\pa_{t}w(s)ds\quad in \; H^{-1}. \label{bi:1}
\end{align}
We take $0 < \e < T' <T$. For all $t_{0} \in [\e, T']$ and sufficiently small $h \in \R$,
\begin{align*}
\left\|\frac{w(t_{0}+h)-w(t_{0})}{h}-\pa_{t}w(t_{0})\right\|_{H^{-1}} &\leq \frac{1}{|h|} \left|\int_{t_{0}}^{t_{0}+h}\|\pa_{t}w(s)-\pa_{t}w(t_{0})\|_{H^{-1}}ds\right| \\
&\leq \sup_{|s-t_{0}| \leq |h| }\|\pa_{t}w(s)-\pa_{t}w(t_{0})\|_{H^{-1}}.
\end{align*}
Since $t \mapsto \pa_{t}u(t) \in H^{-1}(\R^{n})$ is uniformly continuous on $[0,T]$, we obtain (i). \\
\underline{Proof of (ii).} \quad Since $W^{1,q'}(\R^{n}) \hookrightarrow L^{2}(\R^{n})$ and $w \in C([0,T], H^{-1})$ and $\f$ satisfies $(\textbf{H}_{\f})$, we clearly get 
\begin{align}
\D w \in L^{p}([0,T], W^{-1,q}) \quad \text{and}\quad \D\f \in L^{p}([0,T], W^{-1,q}). \label{td23}
\end{align} 
Moreover, using Sobolev embedding and duality argument, we conclude $L^{2} \hookrightarrow W^{-1,q}$. Thus Lemma \ref{lem01} yields
\begin{align}
F(w) \in L^{p}([0,T],W^{-1,q}). \label{td26}
\end{align}
Therefore, concatenating (\ref{td1}), (\ref{td23}) and (\ref{td26}), we obtain  
\[
\pa_{t}w \in L^{p}([0,T],W^{-1,q}).
\]
Let $t_{0} \in [0,T]$. By (\ref{bi:1}), for any $t \in [0,T]$, 
\[
w(t)-w(t_{0})=\int_{t_{0}}^{t}\pa_{t}w(s)ds \quad in\ {\mathcal S}'(\R^{n}),
\]
where ${\mathcal S}(\R^{n})$ and ${\mathcal S}'(\R^{n})$ denote Schwartz space on $\R^{n}$ and the space of tempered distributions on $\R^{n}$, respectively. Using H{\"o}lder's inequality, we get 
\begin{align*}
\left\|\int_{t_{0}}^{\cdot}\pa_{t}w(s)ds\right\|_{L^{p}([0,T],W^{-1,q})} &\leq \left\|\int_{t_{0}}^{\cdot}\|\pa_{t}w(s)\|_{W^{-1,q}}ds\right\|_{L^{p}([0,T])} \\
&\leq \left[\int_{0}^{T}(t-t_{0})^{p/p'}\left(\int_{t_{0}}^{t}\|\pa_{t}w(s)\|_{W^{-1,q}}^{p}ds\right) dt\right]^{1/p} \\
&\leq T^{1/p'} \left[\int_{0}^{T}\left(\int_{t_{0}}^{t}\|\pa_{t}w(s)\|_{W^{-1,q}}^{p}ds\right) dt\right]^{1/p} \\
&\leq T^{1/p'} \left( T \times \int_{0}^{T}\|\pa_{t}w(s)\|_{W^{-1,q}}^{p}ds\right)^{1/p} \\
&\leq T\|\pa_{t}w\|_{L^{p}_{T}W^{-1,q}}.
\end{align*}
Therefore, for all $t_{0} \in [0,T]$, 
\begin{align}
w(\cdot)-w(t_{0})=\int_{t_{0}}^{\cdot}\pa_{t}w(s)ds\quad in \quad L^{p}([0,T],W^{-1,q}). \label{td27}
\end{align}
Combining (\ref{td27}) with Strichartz's estimate, in a way similar to the preceding argument, for all $0 < \e < T' < T$, we obtain
\begin{align*}
&\left\|\frac{w(\cdot+h)-w(\cdot)}{h}-\pa_{t}w(\cdot)\right\|_{L^{p}([\e, T'],W^{-1,q})} \\
&\leq \left\|\frac{1}{h}\int_{\cdot}^{\cdot+h}\|\pa_{t}w(s)-\pa_{t}w(\cdot)\|_{W^{-1,q}}ds\right\|_{L^{p}([\e,T'])} \\
&= \left\{\int_{\e}^{T'}\left|\frac{1}{h}\int_{t_{0}}^{t_{0}+h}\|\pa_{t}w(s)-\pa_{t}w(t_{0})\|_{W^{-1,q}}ds\right|^{p}dt_{0}\right\}^{1/p} \\
&\leq h^{1/p'-1}\left\{\int_{\e}^{T'}\left(\int_{t_{0}}^{t_{0}+h}\|\pa_{t}w(s)-\pa_{t}w(t_{0})\|_{W^{-1,q}}^{p}ds\right)dt_{0}\right\}^{1/p} \\
&=h^{1/p'-1}\left\{\int_{0}^{h}\left(\int_{\e}^{T'}\|\pa_{t}w(t_{0}+s)-\pa_{t}w(t_{0})\|_{W^{-1,q}}^{p}dt_{0}\right)ds\right\}^{1/p}  \\
&\leq \sup_{0 \leq s \leq h}\left(\int_{\e}^{T'}\|\pa_{t}w(t_{0}+s)-\pa_{t}w(t_{0})\|_{W^{-1,q}}^{p}dt_{0}\right)^{1/p} \\
&\rightarrow 0\quad as \quad h \rightarrow 0.
\end{align*}
This completes the proof of Lemma \ref{delt:1}.
\end{proof}

\section{The proof of the main result}
Since Schr{\"o}dinger operator $e^{it\D}$ becomes bounded operator from $\f+H^{1}$ to itself (see Lemma 3 in G{\'e}rard \cite{bib001}), we can obtain 
\[
\f =e^{it\D}\f -i\int_{0}^{t}e^{i(t-t')\D}\D \f dt'.
\]
Combining the above equality with (\ref{eq:01}), we get
\begin{align}
\f+w(t) =e^{it\D}(\f+w_{0})-i\int_{0}^{t}e^{i(t-t')\D}\wt{F}(w(t'))dt', \label{ieq2}
\end{align}
where $\wt{F}(w):= -f(|\f+w|^{2})(\f+w)$. From now on, we deduce the proof in a way similar to Ozawa \cite{bib002}. Acting $\n$ on (\ref{ieq2}), we obtain  
\begin{align}
&||\n (\f+w(t))||_{L^{2}}^{2} \nonumber \\
&=||\n e^{i(-t)\D}(\f+w(t))||_{L^{2}}^{2} \nonumber \\
&=||\n (\f+w_{0})||_{L^{2}}^{2} -2\im \left(\n (\f+w_{0}),\int_{0}^{t}e^{i(-t')\D}\n \wt{F}(w(t'))dt'\right)_{L^{2}} \nonumber \\
&\qquad +\left\|\int_{0}^{t}e^{i(-t')\D}\n \wt{F}(w(t'))dt' \right\|^{2}_{L^{2}}. \label{term:01}
\end{align}
The second term on the RHS of (\ref{term:01}) satisfies the following equality: 
\begin{align}
&-2\im\left(\n (\f+w_{0}),\int_{0}^{t}e^{i(-t')\D}\n \wt{F}(u(t'))dt'\right)_{L^{2}} \nonumber \\
&\qquad  = -2\im\int_{0}^{t}\langle e^{it'\D}\n (\f+w_{0}), \overline{\n (\wt{F}(w(t')))}\rangle dt', \label{term02}
\end{align}
where the time integral of the scalar product is understood as the duality coupling on $(L_{T}^{1}L^{2}\cap L_{t}^{p}L^{q})\times (L^{\infty}_{T}L^{2}+L^{p'}_{T}L^{q'})$ with $(p,q)=(6/n,6)$ if $n=2$, $3$, $(p,q)=(2,4)$ if $n=4$.
For the last term on the RHS of (\ref{term:01}), Fubini's theorem implies
\begin{align}
&\left\|\int_{0}^{t}e^{i(-t')\D}\n (\wt{F}(w(t')))dt' \right\|^{2}_{L^{2}} \nonumber \\
&\qquad =2\re\int_{0}^{t}\left\langle \n (\wt{F}(w(t'))),\overline{\int_{0}^{t'}e^{i(t'-t'')\D}\n (\wt{F}(w(t'')))dt''} \right\rangle dt', \label{term:03}
\end{align}
where the time integral of the scalar product is understood as the duality coupling on $(L^{\infty}_{T}L^{2}+L^{p'}_{T}L^{q'})\times (L_{T}^{1}L^{2}\cap L_{T}^{p}L^{q})$. Concatenating (\ref{term:01}) - (\ref{term:03}), we compute
\begin{align*}
&\|\n (\f+w(t))\|_{L^{2}}^{2} \\
&=||\n (\f+w_{0})||_{L^{2}}^{2}-2\im \int_{0}^{t}\langle e^{it'\D}\n (\f+w_{0}), \overline{\n (\wt{F}(w(t')))}\rangle dt' \\
&\qquad +2\re\int_{0}^{t}\left\langle \n (\wt{F}(w(t'))),\overline{\int_{0}^{t'}e^{i(t'-t'')\D}\n (\wt{F}(w(t'')))dt''} \right\rangle dt'  \\
&=||\n (\f+w_{0})||_{L^{2}}^{2}+2\im \int_{0}^{t}\langle \n (\wt{F}(w(t')),\overline{ e^{it'\D}\n (\f+w_{0})}\rangle dt' \\
&\qquad +2\im\int_{0}^{t}\left\langle \n (\wt{F}(w(t'))),\overline{-i\int_{0}^{t'}e^{i(t'-t'')\D}\n (\wt{F}(w(t'')))dt''} \right\rangle dt'  \\
&=||\n (\f+w_{0})||_{L^{2}}^{2}+\lim_{\e \downarrow 0}2\im \int_{0}^{t}\langle (1-\e\D)^{-1}\n (\wt{F}(w(t'))), \overline{\n w(t')}\rangle dt', 
\end{align*}
where the last equality in the above holds by using (\ref{eq:01}). Taking the duality coupling between the equation (\ref{gp1}) and $(1-\e\D)^{-1}\n (\wt{F}(w))$ on $H^{-1}\times H^{1}$ and using $\im\{\langle (1-\e\D)^{-1}\wt{F}(w),\overline{\wt{F}(w)}\rangle \}=0$, we obtain 
\begin{align*}
\im \langle (1-\e\D)^{-1}\n (\wt{F}(w)),\overline{\n w} \rangle &= \im\{-i\langle (1-\e\D)^{-1}\wt{F}(w),\overline{\pa_{t}w}\rangle\}. 
\end{align*}
From these equalities, we can show
\begin{align}
&\|\n (\f+w(t))\|_{L^{2}}^{2} \nonumber \\
&=||\n (\f+w_{0})||_{L^{2}}^{2}-\lim_{\e \downarrow 0}2\re \int_{0}^{t}\langle (1-\e\D)^{-1}F(u(t'), \overline{\pa_{t}w(t')}\rangle dt' \nonumber \\
&=||\n (\f+w_{0})||_{L^{2}}^{2}-2\re \int_{0}^{t}\langle F(w(t'), \overline{\pa_{t}w(t')}\rangle dt'. \label{5eq:01}
\end{align}
Note that in the above time integral of the scalar product in the last line is understood as the duality coupling on $(L^{\infty}_{T}H^{1}+(L^{\infty}_{T}H^{1}+L^{p'}_{T}W^{1,q'})) \times ((L_{T}^{1}H^{-1})\cap(L_{T}^{1}H^{-1} \cap L_{T}^{p}W^{-1,q}))$ by applying the idea used Lemma 3 in G{\'e}rard \cite{bib001}, that is, we decompose $F(w)$ as
\[
F(w)=\c(D_{x})F(w)+\sum_{j=1}^{n}(1-\c(D_{x}))P_{j}(D_{x})\pa_{x_{j}}F(w),
\]
where $\c \in C_{0}^{\infty}(\R^{n})$ is a cutoff function such that $0 \leq \c \leq 1$, $\c(\xi)=1$\ for $|\xi| \leq 1$ and $\c(\xi)=0$\ for $|\xi|\geq 2$, and $P_{j}(\xi)=i\xi_{j}/|\xi|^{2}$. 

We show (\ref{5eq:01}). It follows from Theorem \ref{thm:2}, Lemma \ref{lem01}, \ref{lem02}, \ref{lemf05} and \ref{delt:1} that
\begin{align}
&\left|\int_{0}^{t}\langle \wt{F}(w(t')), \overline{\pa_{t}w(t')}\rangle dt'\right| \nonumber \\
&\leq \int_{0}^{t}|\langle \c(D_{x})\wt{F}(w(t')),\overline{\pa_{t}w(t')}\rangle | dt'+ \int_{0}^{t}|\langle (1-\c(D_{x}))\wt{F}(w(t')),\overline{\pa_{t}w(t')}\rangle | dt' \nonumber \\
&\leq \|\c(D_{x})\wt{F}(w)\|_{L^{\infty}_{T}H^{1}}\|\pa_{t}w\|_{L^{1}_{T}H^{-1}} \nonumber \\
&\qquad +\|(1-\c(D_{x}))\wt{F}(w)\|_{L^{p'}_{T}(H^{1}+W^{1,q'})}\|\pa_{t}w\|_{L^{p}_{T}(H^{-1} \cap W^{-1,q})}. \nonumber \\
&\leq C(\|F_{1}(w)\|_{L^{\infty}_{T}L^{2}}+\|F_{2}(w)\|_{L^{\infty}_{T}L^{q'}})\|\pa_{t}w\|_{L^{\infty}_{T}H^{-1}} \nonumber \\
&\quad +C(\|G_{1}(w)\|_{L^{p'}_{T}L^{2}}+\|G_{2}(w)\|_{L^{p'}_{T}L^{q'}})\|\pa_{t}w\|_{L^{p}_{T}(H^{-1} \cap W^{-1,q})}. \label{ti:1}
\end{align}
Furthermore, by using a similar argument to the above and Lebesgue convergence theorem, we deduce that
\begin{align*}
\lim_{\e \downarrow 0}\int_{0}^{t}\langle (1-\e\D)^{-1}\wt{F}(w(t')), \overline{\pa_{t}w(t')}\rangle dt' = \int_{0}^{t}\langle \wt{F}(w(t')), \overline{\pa_{t}w(t')}\rangle dt',
\end{align*}
which yields (\ref{5eq:01}). 

From (\ref{5eq:01}), formally, we can continue as follows:
\begin{align*}
\|\n (\f+w(t))\|_{L^{2}}^{2} &=||\n (\f+w_{0})||_{L^{2}}^{2}- \int_{0}^{t}\frac{\pa}{\pa t}\left(\int_{\R^{n}}V(|\f+w(t')|^{2})dx\right) dt'\\
&=||\n (\f+w_{0})||_{L^{2}}^{2}- \int_{\R^{n}}V(|\f+w(t)|^{2})dx+\int_{\R^{n}}V(|\f|^{2})dx,
\end{align*}
since a formal argument implies 
\[
\frac{\pa}{\pa t}\left(\int_{\R^{n}}V(|\f+w(t)|^{2})dx\right) = 2\re \langle \wt{F}(w(t)), \overline{\pa_{t}w(t)}\rangle.
\]
Hence, to justify the argument above, we need to show the following lemma. 
\begin{lem}\label{llem:1}
$\ds \int_{\R^{n}}V(|\f+ w(\cdot)|^{2})dx \in W^{1,1}((0,T))$ \\
and
\[
\frac{\pa}{\pa t}\left(\int_{\R^{n}}V(|\f+w(t)|^{2})dx\right) = 2\re \langle \wt{F}(w(t)), \overline{\pa_{t}w(t)}\rangle_{A \times B} \quad \text{in} \quad L^{1}((0,T)),
\]
where $A:=(H^{1}+(H^{1}+W^{1,q'}))$, $B:=(H^{-1}\cap(H^{-1} \cap W^{-1,q}))$.
\end{lem}

\begin{proof}
Put $I=(0,T)$ for simplicity. Moreover, ${\mathcal D}(I)$ and ${\mathcal D}'(I)$ denote the Fr{\'e}chet space of $C^{\infty}$ functions $I \rightarrow \C$ compactly supported in $I$ and the space of distributions on $I$, respectively. Note that as is in Gallo \cite{bib007}, from $(\textbf{H}_{f})$, the mapping $w \mapsto V(|\f+w|^{2})$ become a bounded operator from $H^{1}(\R^{n})$ to $L^{1}(\R^{n})$. 
Thus, for any $\v \in C^{\infty}_{0}(0,T)$, we have
\begin{align*}
&\left\langle \frac{\pa}{\pa t}\int_{\R^{n}}V(|\f+w|^{2})dx, \v \right\rangle_{{\mathcal D}'(I) \times {\mathcal D}(I)} \\
&= \left\langle -\int_{\R^{n}}V(|\f+w|^{2})dx, \pa_{t}\v \right\rangle_{{\mathcal D}'(I) \times {\mathcal D}(I)} \\
&=-\int_{I}\left(\int_{\R^{n}}V(|\f+w(t)|^{2})dx \right)\pa_{t}\v(t) dt. 
\end{align*}
Take $0 < \e < T' <T$ such that $\text{supp}(\v) \subset [\e, T']$. Using Lebesgue convergence theorem, we compute
\begin{align*}
&-\int_{I}\left(\int_{\R^{n}}V(|\f+w(t)|^{2})dx \right)\pa_{t}\v(t) dt \\
&=\lim_{h \rightarrow 0}\left\{ -\int_{\e}^{T'}\left(\int_{\R^{n}}V(|\f+w(t)|^{2})dx \right)\frac{\v(t+h)-\v(t)}{h} dt \right\} \\
&=\lim_{h' \rightarrow 0}\left\{ \int_{\e}^{T'}\left(\int_{\R^{n}}\frac{V(|\f+w(t+h')|^{2})-V(|\f+w(t)|^{2})}{h'} dx\right)\v(t) dt \right\} \\
&=\int_{I} 2\re \langle F(u(t)), \overline{\pa_{t}w(t)}\rangle_{A \times B} \v(t) dt.
\end{align*}
We need to justify the limiting procedure of the last line in the above. Since $(\pa /\pa \bar{z})(V(|z|^{2})) = \wt{F}(|z|)$ for any $z \in \C$, it follows that
\begin{align}
&\left|\int_{\R^{n}}\frac{V(|\f+w(t+h)|^{2})-V(|\f+w(t)|^{2})}{h}dx-2\re \langle \wt{F}(w(t)),\overline{\pa_{t}w(t)}\rangle_{A \times B} \right| \nonumber \\
&\leq \left|\int_{\R^{n}}2\re \left( \int_{0}^{1}\frac{\pa V}{\pa \bar{z}}(|\f+w(t)+\t(w(t+h)-w(t))|^{2})d\t \right. \right. \nonumber \\ 
&\qquad \quad \times \left. \frac{\overline{(w(t+h)-w(t))}}{h}\right) dx  \left. - 2\re\langle \wt{F}(w(t)),\overline{\pa_{t}w(t)}\rangle_{A \times B} \right| \nonumber \\
&\leq 2\left|\int_{\R^{n}}\left( \int_{0}^{1}\left(\frac{\pa V}{\pa \bar{z}}(|\f+w(t)+\t(w(t+h)-w(t))|^{2})-\wt{F}(w(t))\right) d\t \right. \right. \nonumber \\
&\qquad \quad \left. \left. \times \frac{\overline{(w(t+h)-w(t))}}{h}\right) dx\right| \nonumber \\
&\qquad +2\left|\int_{\R^{n}}\wt{F}(w(t))\frac{\overline{(w(t+h)-w(t))}}{h}dx- \langle \wt{F}(w(t)),\overline{\pa_{t}w(t)}\rangle_{A \times B}\right| \nonumber \\
&\leq 2\left|\int_{\R^{n}}\left( \int_{0}^{1}\left(\wt{F}( w(t)+\t(w(t+h)-w(t)))-\wt{F}(w(t))\right)d \t \right. \right. \nonumber \\
&\qquad \quad \left. \left. \times \frac{\overline{(w(t+h)-w(t))}}{h}\right) dx\right| \nonumber \\
&\qquad +2\left|\left\langle \wt{F}(w(t)),\frac{\overline{(w(t+h)-w(t))}}{h}\right\rangle_{H^{-1} \times H^{1}} - \langle \wt{F}(w(t)),\overline{\pa_{t}w(t)}\rangle_{A \times B} \right| \nonumber \\
&=: 2L_{1}+2L_{2}. \label{last:1}
\end{align}
\underline{The estimation of $L_{1}$.}\quad Choose the cutoff function $\c \in C_{0}^{\infty}(\R^{n})$ such that $0 \leq \c \leq 1$, $\c(\xi)=1$ for $|\xi|\leq 1$ and $\c(\xi)=0$ for $|\xi|\geq 2$. Using $\c(D_{x})$, we decompose $L_{1}$ as follows:
\begin{align*}
L_{1}&\leq \left|\int_{\R^{n}}\left( \int_{0}^{1}\c(D_{x})\left\{\wt{F}( w(t)+\t(w(t+h)-w(t)))-\wt{F}(w(t))\right\} d\t \right. \right. \nonumber \\
&\qquad \quad \left. \left. \times \frac{\overline{(w(t+h)-w(t))}}{h}\right) dx\right| \\
&\quad +\left|\int_{\R^{n}}\left( \int_{0}^{1}(1-\c(D_{x}))\left\{\wt{F}( w(t)+\t(w(t+h)-w(t)))-\wt{F}(w(t))\right\} d\t \right. \right. \nonumber \\
&\qquad \quad \left. \left. \times \frac{\overline{(w(t+h)-w(t))}}{h}\right) dx\right| \\
&=:K_{1}+K_{2}. 
\end{align*}
From now on, $L^{p}_{[\e, T']}X$ denotes the Banach space $L^{p}([\e,T'], X)$ for $p \in [1, \infty]$ and a Banach space $X$. \\
\underline{The estimation of $K_{1}$.}\quad  By Lemma \ref{lem03}, we get
\begin{align}
&\|F(w( \cdot )+\t(w(\cdot +h)-w(\cdot)))-F(w(\cdot))\|_{L^{\infty}_{[\e, T']}L^{2}+L^{\infty}_{[\e, T']}L^{q'}} \nonumber \\
 &\leq \|F_{1}(w( \cdot )+\t(w(\cdot +h)-w(\cdot)))-F_{1}(w(\cdot))\|_{L^{\infty}_{[\e, T']}L^{2}} \nonumber \\
  &\quad +\|F_{2}(w( \cdot )+\t(w(\cdot +h)-w(\cdot)))-F_{2}(w(\cdot))\|_{L^{\infty}_{[\e, T']}L^{q'}} \nonumber \\
 &\leq C\|w(\cdot +h)-w(\cdot)\|_{L^{\infty}_{[\e, T']}L^{2}}+C\|w(\cdot +h)-w(\cdot)\|_{L^{\infty}_{[\e, T']}H^{1}} \nonumber \\
  &\quad \times \{ (\|w(\cdot +h)\|_{L^{\infty}_{[\e, T']}H^{1}}+\|w(\cdot)\|_{L^{\infty}_{[\e, T']}H^{1}}) \nonumber \\
  &\quad \quad +(\|w(\cdot +h)\|_{L^{\infty}_{[\e, T']}H^{1}}+\|w(\cdot)\|_{L^{\infty}_{[\e, T']}H^{1}})^{\max(1,2\a_{1}-2)} \} \nonumber \\
 &\leq C\|w(\cdot +h)-w(\cdot)\|_{L^{\infty}_{[\e, T']}H^{1}}+C\|w(\cdot +h)-w(\cdot)\|_{L^{\infty}_{[\e, T']}H^{1}} \nonumber \\
  &\quad \times (\|w\|_{L^{\infty}_{T}H^{1}}+\|w\|_{L^{\infty}_{T}H^{1}}^{\max(1,2\a_{1}-2)}) \nonumber \\
 &\leq C\|w(\cdot +h)-w(\cdot)\|_{L^{\infty}_{[\e, T']}H^{1}}(1+\|w\|_{L^{\infty}_{T}H^{1}}+\|w\|_{L^{\infty}_{T}H^{1}}^{\max(1,2\a_{1}-2)}) \nonumber \\
 &\leq C\|w(\cdot +h)-w(\cdot)\|_{L^{\infty}_{[\e, T']}H^{1}}, \label{Furi1}
\end{align}
where $C$ depends on the norm $\|w\|_{X_{T}}$ of the space $X_{T}$. \\
$X_{[\e, T']}$ denotes $L^{\infty}_{[\e, T']}H^{1} \cap L^{p}_{[\e, T']}W^{1,q}$. Using the estimate similar to (\ref{lem05}) and (\ref{Furi1}), we obtain
\begin{align*}
 &\int_{\e}^{T'} K_{1} dt \\
 &\leq \int_{\e}^{T'} \left( \int_{0}^{1}\left\|\c(D_{x})\left\{\wt{F}( w(t)+\t(w(t+h)-w(t)))-\wt{F}(w(t))\right\}\right\|_{H^{1}} d\t \right. \\
 &\qquad \left. \times \left\|\frac{\overline{w(t+h)-w(t)}}{h}\right\|_{H^{-1}}\right) dt \\
 &\leq C\int_{0}^{1} \left( \|\wt{F}_{1}( w( \cdot )+\t(w(\cdot +h)-w(\cdot)))-\wt{F}_{1}(w(\cdot))\|_{L^{\infty}_{[\e, T']}L^{2}} \right. \\
 &\qquad \qquad \qquad \left. +\|\wt{F}_{2}( w( \cdot )+\t(w(\cdot +h)-w(\cdot)))-\wt{F}_{2}(w(\cdot))\|_{L^{\infty}_{[\e, T']}L^{q'}} \right) d \t \\
 &\qquad \times \left\|\frac{w(\cdot+h)-w(\cdot)}{h}\right\|_{L^{\infty}_{[\e, T']}H^{-1}} \\
 &\leq C\int_{0}^{1}\|w(\cdot+h)-w(\cdot)\|_{X_{[\e, T']}} d \t \left\|\frac{w(\cdot+h)-w(\cdot)}{h}\right\|_{L^{\infty}_{[\e, T']}H^{-1}} \\
 &\leq C\|w(\cdot+h)-w(\cdot)\|_{X_{[\e, T']}} \left\|\frac{w(\cdot+h)-w(\cdot)}{h}\right\|_{L^{\infty}_{[\e, T']}H^{-1}}.
\end{align*}
By lemma \ref{lem04}, we have
\begin{align}
 &\|\n F(w(\cdot)+\t(w(\cdot+h)-w(\cdot)))-\n F(w(\cdot))\|_{L^{\infty}_{[\e, T']}L^{2}+L^{p'}_{[\e, T']}L^{q'}} \nonumber \\
 &\leq \|\wt{G}_{1}(w(\cdot)+\t(w(\cdot+h)-w(\cdot)))-\wt{G}_{1}(w(\cdot))\|_{L^{\infty}_{[\e, T']}L^{2}} \nonumber \\
 &\quad +\|\wt{G}_{2}(w(\cdot)+\t(w(\cdot+h)-w(\cdot)))-\wt{G}_{2}(w(\cdot))\|_{L^{p'}_{[\e, T']}L^{q'}} \nonumber \\
 & \leq C\|w(\cdot+h)-w(\cdot)\|_{L^{\infty}_{[\e, T']}H^{1}} \nonumber \\
 &\quad +C(1+\|w(\cdot+h)\|_{L^{\infty}_{[\e, T']}H^{1}}+\|w(\cdot)\|_{L^{\infty}_{[\e, T']}H^{1}})^{\max(1,2\a_{1}-2)} \nonumber \\ 
 &\quad \quad \times \|w(\cdot+h)-w(\cdot)\|_{L^{\infty}_{[\e, T']}H^{1}} \nonumber \\
 &\quad +C\|w(\cdot+h)-w(\cdot)\|_{X_{[\e, T']}}(\|w(\cdot+h)\|_{X_{[\e, T']}}^{\max(1,2\a_{1}-2)}+\|w(\cdot)\|_{X_{[\e, T']}}^{\max(1,2\a_{1}-2)}) \nonumber \\
 &\quad +C\|w(\cdot+h)-w(\cdot)\|_{X_{[\e, T']}}(1+\|w(\cdot+h)\|_{L^{\infty}_{[\e, T']}H^{1}}+\|w(\cdot)\|_{L^{\infty}_{[\e, T']}H^{1}}) \nonumber \\
 &\quad \quad \times (\|w(\cdot+h)\|_{X_{[\e, T']}}^{\max(0,2\a_{1}-3)}+\|w(\cdot)\|_{X_{[\e, T']}}^{\max(0,2\a_{1}-3)}) \nonumber \\
 &\leq C\|w(\cdot+h)-w(\cdot)\|_{X_{[\e, T']}} \label{Furi2},
\end{align}
where $C$ depends on $\|w\|_{X_{T}}$. Using the estimate similar to (\ref{lem06}) and (\ref{Furi2}), we have
\begin{align*}
 &K_{2} \leq \left\|\int_{0}^{1}(1-\c(D_{x}))\left\{\wt{F}( w(t)+\t(w(t+h)-w(t)))-\wt{F}(w(t))\right\} d\t \right\|_{H^{1}+W^{1,q'}} \\
 &\quad \quad \times \left\|\frac{\overline{(w(t+h)-w(t))}}{h}\right\|_{H^{-1} \cap W^{-1,q}} \\
 &\leq \int_{0}^{1} \left\|(1-\c(D_{x}))\left\{\wt{F}( w(t)+\t(w(t+h)-w(t)))-\wt{F}(w(t))\right\} \right\|_{H^{1}+W^{1,q'}} d \t \\
 &\quad \quad \times \left\|\frac{\overline{(w(t+h)-w(t))}}{h}\right\|_{H^{-1} \cap W^{-1,q}} \\
 &\leq C\int_{0}^{1}\left( \|\wt{G}_{1}( w(t)+\t(w(t+h)-w(t)))-\n \wt{G}_{1}(w(t))\|_{L^{2}} \right. \\
 &\qquad \qquad \qquad \left. +\|\wt{G}_{2}( w(t)+\t(w(t+h)-w(t)))-\n \wt{G}_{2}(w(t))\|_{L^{q'}} \right) d \t \\
 &\quad \quad \times \left\|\frac{w(t+h)-w(t)}{h}\right\|_{H^{-1} \cap W^{-1,q}}. 
\end{align*}
Hence, we get
\begin{align*}
&\int_{\e}^{T'}K_{2} dt \\
 &\leq C \left\| \int_{0}^{1} \left( \|\wt{G}_{1}( w(\cdot)+\t(w(\cdot+h)-w(\cdot)))-\wt{G}_{1}(w(\cdot))\|_{L^{2}} \right. \right. \\
 &\qquad \qquad \qquad \left. \left. +\|\wt{G}_{2}( w(\cdot)+\t(w(\cdot+h)-w(\cdot)))-\wt{G}_{2}(w(\cdot))\|_{L^{q'}} d\t \right) \right\|_{L^{p'}_{[\e, T']}} \\
 &\qquad \qquad \times \left\|\frac{w(\cdot+h)-w(\cdot)}{h}\right\|_{L^{p}_{[\e, T']}(H^{-1} \cap W^{-1,q})} \\
 &\leq C \int_{0}^{1} \left( \|\wt{G}_{1}( w(\cdot)+\t(w(\cdot+h)-w(\cdot)))-\wt{G}_{1}(w(\cdot))\|_{L^{\infty}_{[\e, T']}L^{2}} \right. \\
 &\qquad \qquad \qquad \left. +\|\wt{G}_{2}( w(\cdot)+\t(w(\cdot+h)-w(\cdot)))-\wt{G}_{2}(w(\cdot))\|_{L^{p'}_{[\e, T']}L^{q'}} \right)  d \t \\
 &\qquad \qquad \times \left\|\frac{w(\cdot+h)-w(\cdot)}{h}\right\|_{L^{\infty}_{[\e, T']}H^{-1} \cap L^{p}_{[\e, T']}W^{-1,q}} \\
 &\leq C \int_{0}^{1}  \|w(\cdot+h)-w(\cdot)\|_{X_{[\e, T']}} d \t \left\|\frac{w(\cdot+h)-w(\cdot)}{h}\right\|_{L^{\infty}_{[\e, T']}H^{-1} \cap L^{p}_{[\e, T']}W^{-1,q}} \\
 &\leq C\|w(\cdot+h)-w(\cdot)\|_{X_{[\e, T']}} \left\|\frac{w(\cdot+h)-w(\cdot)}{h}\right\|_{L^{\infty}_{[\e, T']}H^{-1} \cap L^{p}_{[\e, T']}W^{-1,q}}.  
\end{align*}
Thus, we obtain
\begin{align}
\int_{\e}^{T'}L_{1}dt &= \int_{\e}^{T'} K_{1}dt+\int_{\e}^{T'} K_{2}dt \nonumber \\
&\leq C\|w(\cdot+h)-w(\cdot)\|_{X_{[\e, T']}} \nonumber \\
&\quad \times \left( \left\|\frac{w(\cdot+h)-w(\cdot)}{h}\right\|_{L^{\infty}_{[\e, T']}H^{-1}}+\left\|\frac{w(\cdot+h)-w(\cdot)}{h}\right\|_{L^{p}_{[\e, T']}W^{-1,q}} \right). \label{last:2}
\end{align}
\underline{The estimation of $L_{2}$.}\quad It follows from Lemma \ref{lem01}, Lemma \ref{lem03} and Lemma \ref{lemf05} that for almost all $t \in [\e, T']$, 
\begin{align*}
L_{2} &\leq \left|\left\langle \c(D_{x})\wt{F}(w(t)),\frac{\overline{w(t+h)-w(t)}}{h}-\overline{\pa_{t}w(t)}\right\rangle_{H^{1} \times H^{-1}} \right| \\
&\quad +\left|\left\langle (1-\c(D_{x}))\wt{F}(w(t)), \frac{\overline{w(t+h)-w(t)}}{h}-\overline{\pa_{t}w(t)}\right\rangle_{\begin{subarray}{l} (H^{1}+W^{1,q'}) \\ \times (H^{-1} \cap W^{-1,q}) \end{subarray}} \right| \\
&\leq \|\c(D_{x})\wt{F}(w(t))\|_{H^{1}}\left\|\frac{w(t+h)-w(t)}{h}-\pa_{t}w(t)\right\|_{H^{-1}} \\
&\quad +\|(1-\c(D_{x}))\wt{F}(w(t))\|_{H^{1}+W^{q'}}\left\|\frac{(w(t+h)-w(t))}{h}-\pa_{t}w(t)\right\|_{H^{-1} \cap W^{-1,q}} \\
&\leq C\left( \|\wt{F}_{1}(w(t))\|_{L^{2}}+ \|\wt{F}_{2}(w(t))\|_{L^{q'}} \right) \left\|\frac{w(t+h)-w(t)}{h}-\pa_{t}w(t)\right\|_{H^{-1}} \\
&\quad +C\left( \|\wt{G}_{1}(w(t))\|_{L^{2}}+ \|\wt{G}_{2}(w(t))\|_{L^{q'}} \right) \\
&\quad \quad \times \left( \left\|\frac{w(t+h)-w(t)}{h}-\pa_{t}w(t)\right\|_{H^{-1}} \right. \\
& \qquad \quad \quad \left. +\left\|\frac{w(t+h)-w(t)}{h}-\pa_{t}w(t)\right\|_{W^{-1,q}} \right).
\end{align*}
Hence, we deduce that
\begin{align}
&\int_{\e}^{T'}L_{2}dt \nonumber \\
&\leq C\left( \|\wt{F}_{1}(w)\|_{L^{1}_{T}L^{2}}+ \|\wt{F}_{2}(w)\|_{L^{1}_{T}L^{q'}} \right) \left\|\frac{w(\cdot+h)-w(\cdot)}{h}-\pa_{t}w(\cdot)\right\|_{L^{\infty}_{[\e, T']}H^{-1}} \nonumber \\
 &\quad +C\left( \|\wt{G}_{1}(w)\|_{L^{p'}_{T}L^{2}}+ \|\wt{G}_{2}(w)\|_{L^{p'}_{T}L^{q'}} \right) \nonumber \\
  &\quad \quad \times \left( \left\|\frac{w(\cdot+h)-w(\cdot)}{h}-\pa_{t}w(\cdot)\right\|_{L^{\infty}_{[\e, T']}H^{-1}} \right. \nonumber \\ 
& \qquad \quad \quad \left. +\left\|\frac{w(\cdot+h)-w(\cdot)}{h}-\pa_{t}w(\cdot)\right\|_{L^{p}_{[\e, T']}W^{-1,q}} \right). \label{last:3}
\end{align}
In conclusion, concatenating (\ref{last:1}), (\ref{last:2}) and (\ref{last:3}),
\begin{align*}
&\left| \int_{\e}^{T'}\int_{\R^{n}}\left(\frac{V(|\f+w(t+h')|^{2})-V(|\f+w(t)|^{2})}{h'} \right)\v(t) dtdx \right. \\
&\qquad \left. -\int_{\e}^{T'}(\re \langle \wt{F}(w(t)), \overline{\pa_{t}w(t)}\rangle)\v(t) dt \right| \\
&\leq 2\int_{\e}^{T'}L_{1}|\v(t)|dt + 2\int_{\e}^{T'}L_{2}|\v(t)|dt \\
&\leq C\|w(\cdot+h)-w(\cdot)\|_{X_{[\e, T']}} \\
&\qquad \times \left( \left\|\frac{w(\cdot+h)-w(\cdot)}{h}\right\|_{L^{\infty}_{[\e, T']}H^{-1}}+\left\|\frac{w(\cdot+h)-w(\cdot)}{h}\right\|_{L^{p}_{[\e, T']}W^{-1,q}} \right) \\
&\quad +C\left( \|\wt{F}_{1}(w)\|_{L^{1}_{T}L^{2}}+ \|\wt{F}_{2}(w)\|_{L^{1}_{T}L^{q'}} \right) \left\|\frac{w(\cdot+h)-w(\cdot)}{h}-\pa_{t}w(\cdot)\right\|_{L^{\infty}_{[\e, T']}H^{-1}} \\
 &\quad +C\left( \|\wt{G}_{1}(w)\|_{L^{p'}_{T}L^{2}}+ \|\wt{G}_{2}(w)\|_{L^{p'}_{T}L^{q'}} \right) \\
 &\qquad \times \left( \left\|\frac{w(\cdot+h)-w(\cdot)}{h}-\pa_{t}w(\cdot)\right\|_{L^{\infty}_{[\e, T']}H^{-1}} \right. \\
 &\qquad \quad \quad \left. +\left\|\frac{w(\cdot+h)-w(\cdot)}{h}-\pa_{t}w(\cdot)\right\|_{L^{p}_{[\e, T']}W^{-1,q}} \right). 
\end{align*}
Noting Lemma \ref{delt:1} and the fact that a local Lipschitz continuity (\ref{thm2f}) in Theorem \ref{thm:1} yields 
\begin{align*}
\|w(\cdot+h)-w(\cdot)\|_{X_{[\e, T']}} &\leq C\|w(|h|)-w_{0}\|_{H^{1}} \\
 &\rightarrow 0 \quad as\quad h \rightarrow 0,
\end{align*}
we obtain 
\[
\left\langle \frac{\pa}{\pa t}\int_{\R^{n}}V(u)dx, \v \right\rangle_{{\mathcal D}' \times {\mathcal D}(I)}=
2\int_{I} \re \langle \wt{F}(w(t)), \overline{\pa_{t}w(t)}\rangle \v(t) dt. 
\]
Since the estimation (\ref{ti:1}) means $\re \langle \wt{F}(w(t)), \overline{\pa_{t}w(t)}\rangle \in L^{1}(I)$, we complete the proof of Lemma \ref{llem:1}.
\end{proof}
In conclusion, by Lemma \ref{llem:1}, we complete the proof of the main result. $\square$

\section*{Acknowledgment}
The author would like to express deep gratitude to Professor Mishio Kawashita for helpful comments
and warm encouragements.

\section{Appendix}
\begin{proof}[Proof of Lemma \ref{lem01}]
We decompose $-f(|\f+w|^{2})(\f+w)$ as
\begin{align}
-f(|\f+w|^{2})(\f+w) &=\wt{F}_{1}(w)+\wt{F}_{2}(w), \label{ap:1} 
\end{align}
where
\begin{align*}
\wt{F}_{1}(w) &:= -f(|\f|^{2})(\f+w)-2\re[\bar{\f}w] f'(|\f|^{2})\f, \\
\wt{F}_{2}(w) &:= -\{f(|\f+w|^{2})(\f+w)-f(|\f|^{2})(\f+w)\}+2\re[\bar{\f}w] f'(|\f|^{2})\f.
\end{align*}
According to Lemma 4.1 in Gallo \cite{bib007}, by the assumption $(\textbf{H}'_{\f})$ and $f(|\f|^{2}) \in L^{2}$, we deduce that 
\begin{align*}
 \|\wt{F}_{1}(w)\|_{L^{\infty}_{T}L^{2}} \leq C(1+\|w\|_{L^{\infty}_{T}L^{2}}),\quad |\wt{F}_{2}(w)| \leq C|w|^{2}(1+|w|)^{\max(0, 2\a_{1}-3)},
\end{align*}
Therefore for all $t \in [0,T]$, we estimate that
\begin{align*}
 \|\wt{F}_{2}(w(t))\|_{L^{q'}} &\leq C\||w(t)|^{2}(1+|w(t)|)^{\max( 0, 2\a_{1}-3 )} \|_{L^{q'}} \\
  &\leq C\|w(t)\|_{L^{2q'}}^{2}+C\|w(t)\|^{\max(2, 2\a_{1}-1 )}_{L^{q'\max(2, 2\a_{1}-1 )}} \\
  &\leq C\|w(t)\|_{H^{1}}^{2}+C\|w(t)\|^{\max(2, 2\a_{1}-1 )}_{H^{1}}. \\
\end{align*}
Hence, we deduce that
\begin{align*}
\|\wt{F}_{2}(w)\|_{L^{\infty}_{T}L^{q'}} \leq C\|w\|_{L^{\infty}_{T}H^{1}}^{2}+C\|w\|^{\max(2, 2\a_{1}-1 )}_{L^{\infty}_{T}H^{1}}.
\end{align*}
In conclusion, we get
\begin{align*}
&\|\wt{F}_{1}(w)\|_{L^{\infty}_{T}L^{2}}+\|\wt{F}_{2}(w)\|_{L^{\infty}_{T}L^{q'}}  \nonumber \\
&\quad  \leq C(1+\|w\|_{L^{\infty}_{T}L^{2}})+C(\|w\|_{L^{\infty}_{T}H^{1}}^2+\|w\|_{L^{\infty}_{T}H^{1}}^{\max(2, 2\a_{1}-1)}).
\end{align*}

Next, we show $\wt{F}(w) \in L^{p}_{T}L^{2}$. We apply an interpolation method (see Lemma 4.2 in Gallo \cite{bib007}). Thanks to the H{\"o}lder inequality and Gagliardo-Nirenberg's inequality, we estimate 
\begin{align}
\|\wt{F}_{2}(w)\|_{L^{p}_{T}L^{2}} &\leq C\|w\|^{2}_{L^{2p}_{T}L^{4}}+\|w\|^{\max(2, 2\a_{1}-1)}_{L_{T}^{p\max(2, 2\a_{1}-1)}L^{2\max(2, 2\a_{1}-1)}} \nonumber \\
 &\leq C\|w\|^{2}_{L^{\infty}_{T}H^{1}}+\|w\|^{\max(2, 2\a_{1}-1)}_{L_{T}^{s}W^{1,r}}, \label{td24} 
\end{align}
where we choose the pair $(s,r)$ such that
\begin{itemize}
 \item If $\frac{1}{2}-\frac{1}{n} \leq \frac{1}{p\max(2, 2\a_{1}-1)}$ (which means that $H^{1} \hookrightarrow L^{p\max(2, 2\a_{1}-1)}$),\\
         \quad then $(s,r)=(\infty, 2)$.
 \item If $\frac{1}{2}-\frac{1}{n} > \frac{1}{p\max(2, 2\a_{1}-1)}$, \\
         \quad then $r>2$ and
         \begin{enumerate}
           \item $\frac{2}{s}+\frac{n}{r}=\frac{n}{2}$ \quad (which means that $(s,r)$ is an admissible pair), \\
           \item $0 \leq \frac{1}{r}-\frac{1}{n} \leq \frac{1}{p\max(2, 2\a_{1}-1)}$ (which gives the Sobolev embedding $W^{1,r} \hookrightarrow L^{p\max(2, 2\a_{1}-1)}$), \\
           \item $\frac{1}{p\max(2, 2\a_{1}-1)} \geq \frac{1}{s}$ (which gives $L^{s}_{T} \hookrightarrow L^{p\max(2, 2\a_{1}-1)}_{T}$).
         \end{enumerate}
\end{itemize}
Such the choice of $s$ and $r$ is possible if and only if $s$ and $r$ satisfy the following inequality:
\begin{align}
 \frac{n}{2}-1 \leq  \frac{2+n}{p\max(2, 2\a_{1}-1)}. \label{ic3}
\end{align}
Indeed, if (\ref{ic3}) is true, then it is sufficient to choose
\begin{align*}
\frac{n}{r} \in \left[ \frac{n}{2}-\frac{2}{p\max(2, 2\a_{1}-1)},\; 1+\frac{n}{p\max(2, 2\a_{1}-1)} \right].
\end{align*}
Moreover, since $H^{1} \hookrightarrow L^{p\max(2, 2\a_{1}-1)}$ if $n=2$ or if $n=3$ and $1\leq\a_{1}\leq 2$ or if $n=4$ and $1\leq\a_{1}\leq 3/2$, we consider that $n=3$ and $2<\a_{1}<3$ or $n=4$ and $3/2<\a_{1}<2$. Since $2<r<3$ and $(s,r)$ is an admissible pair, we can choose $\wt{\t} \in (0, 1)$ satisfying  
\[
\frac{1-\wt{\t}}{2}+\frac{\wt{\t}}{q}=\frac{1}{r}, \quad \frac{1-\wt{\t}}{\infty}+\frac{\wt{\t}}{p}=\frac{1}{s}.
\]
Thus, using interpolation method, 
\begin{align}
\|w\|_{L^{s}_{T}W^{1,r}} &\leq C\|w\|_{L^{\infty}_{T}H^{1}}^{1-\wt{\t}}\|w\|_{L^{p}_{T}W^{1,q}}^{\wt{\t}} \nonumber \\
 &\leq C(\|w\|_{L^{\infty}_{T}H^{1}}+\|w\|_{L^{p}_{T}W^{1,q}}) \nonumber \\
 &= C\|w\|_{X_{T}}. \label{td25}
\end{align}
From (\ref{td24}) and (\ref{td25}), we deduce that 
\begin{align*}
\|F_{2}(w)\|_{L^{p}_{T}L^{2}} &\leq C(\|w\|^{2}_{L^{\infty}_{T}H^{1}}+\|w\|^{\max(2, 2\a_{1}-1)}_{X_{T}}).  
\end{align*}
Thus, we get $F(w) \in L^{p}_{T}L^{2}$.
\end{proof}

\begin{proof}[Proof of Lemma \ref{lem03}]
we use the decomposition (\ref{ap:1}) again. As is in Gallo \cite{bib007}, we also have
\begin{align*}
|\wt{F}_{1}(w_{1})-\wt{F}_{1}(w_{2})| &\leq C|w_{1}-w_{2}|, \\
|\wt{F}_{2}(w_{2})-\wt{F}_{2}(w_{2})| &\leq C|w_{1}-w_{2}|(|w_{1}|+|w_{2}|)(1+|w_{1}|+|w_{2}|)^{\max(0,2\a_{1}-3)}.
\end{align*}
Therefore we deduce that
\begin{align*}
&\|\wt{F}_{1}(w_{1})-\wt{F}_{1}(w_{2})\|_{L^{\infty}_{T}L^{2}} \leq C\|w_{1}-w_{2}\|_{L^{\infty}_{T}L^{2}}.
\end{align*}
Moreover let 
\[
(q_{1},q_{2}):=
\begin{cases}
 \ds (2, 3)\quad \text{if $n=2$, or $n=3$ and $\a_{1} \leq 2$}, \\
 \ds \left( \frac{q}{q-1-\max(1, 2\a_{1}-2)}, \frac{q}{\max(1, 2\a_{1}-2)} \right) \\
\hspace{4cm} \text{if $n=3$ and $2< \a_{1} < 3$ or $n=4$},
\end{cases}
\]
with $\frac{1}{q'}=\frac{1}{q_{1}}+\frac{1}{q_{2}}$. Since if $n=3$ and $2< \a_{1} < 3$ or $n=4$, then $H^{1} \hookrightarrow L^{q_{1}}$,
for all $t \in [0,T]$, we estimate
\begin{align*}
&\|\wt{F}_{2}(w_{2}(t))-\wt{F}_{2}(w_{2}(t))\|_{L^{q'}} \\
 &\leq C\|w_{1}(t)-w_{2}(t)\|_{L^{2q'}}(\|w_{1}(t)\|_{L^{2q'}}+\|w_{2}(t)\|_{L^{2q'}}) \\
 &\quad +C\|w_{1}(t)-w_{2}(t)\|_{L^{q_{1}}}\||w_{1}(t)|+|w_{2}(t)|\|^{\max(1, 2\a_{1}-2 )}_{L^{q_{2}\max(1, 2\a_{1}-2 )}} \\
 &\leq C\|w_{1}(t)-w_{2}(t)\|_{L^{2q'}}(\|w_{1}(t)\|_{L^{2q'}}+\|w_{2}(t)\|_{L^{2q'}}) \\
 &\quad +C\|w_{1}(t)-w_{2}(t)\|_{L^{q_{1}}}\||w_{1}(t)|+|w_{2}(t)|\|^{\max(1, 2\a_{1}-2 )}_{L^{q}} \\
 &\leq C\|w_{1}(t)-w_{2}(t)\|_{H^{1}}(\|w_{1}(t)\|_{H^{1}}+\|w_{2}(t)\|_{H^{1}}) \\
 &\quad +C\|w_{1}(t)-w_{2}(t)\|_{H^{1}}(\|w_{1}(t)\|_{H^{1}}+\|w_{2}(t)\|_{H^{1}})^{\max(1, 2\a_{1}-2 )}. 
\end{align*}
In conclusion, we get
\begin{align*}
 &\|\wt{F}(w_{1})-\wt{F}(w_{2})\|_{L^{\infty}_{T}L^{2}+L^{\infty}_{T}L^{q'}} \\
  & \leq CT\|w_{1}-w_{2}\|_{L^{\infty}_{T}L^{2}}+C\|w_{1}-w_{2}\|_{L^{\infty}_{T}H^{1}} \\
  &\quad \times ((\|w_{1}\|_{L^{\infty}_{T}H^{1}}+\|w_{2}\|_{L^{\infty}_{T}H^{1}})+(\|w_{1}\|_{L^{\infty}_{T}H^{1}}+\|w_{2}\|_{L^{\infty}_{T}H^{1}})^{\max(1,2\a_{1}-2)}). 
 \end{align*}
\end{proof}

\end{document}